\documentclass[final]{siamltex1213}

\def\be{\begin{equation}}
\def\ee{\end{equation}}
\def\beaN{\setlength{\arraycolsep}{0.0em}\begin{eqnarray*}}
\def\eeaN{\end{eqnarray*}\setlength{\arraycolsep}{5pt}}
\def\bea{\setlength{\arraycolsep}{0.0em}\begin{eqnarray}}
\def\eea{\end{eqnarray}\setlength{\arraycolsep}{5pt}}

\def\dm{n}

\def\ZZ{\mathbb{Z}}

\def\be{\begin{equation}}
\def\ee{\end{equation}}

\def\dm{n}

\def\ZZ{\mathbb{Z}}

\def\bks{\backslash}

\def\ome{\omega}                

\usepackage{amssymb}
\usepackage{tensor}

\title{Generalizations of the Maillet Determinant} 
\author{Youngmi Hur\footnotemark[2]\ \footnotemark[3] \and Zachary Lubberts\footnotemark[3]}

\begin{document}
\maketitle
\renewcommand{\thefootnote}{\fnsymbol{footnote}}
\footnotetext[2]{Department of Mathematics, Yonsei University, Seoul 120-749, Korea.}
\footnotetext[3]{Department of Applied Mathematics and Statistics, Johns Hopkins University, Baltimore, MD 21218, USA.}
\renewcommand{\thefootnote}{\arabic{footnote}}
\footnotetext[1]{This research was partially supported by Yonsei New Faculty Research Seed Money Grant and NSF Grant DMS-1115870.}

\begin{abstract}
We consider several extensions of the Maillet determinant studied by Malo, Turnbull, and Carlitz and Olson, and derive properties of the underlying matrices. In particular, we compute the eigenvectors and eigenvalues of these matrices, which yield formulas for these new determinants. 
\end{abstract}



\pagestyle{myheadings}
\thispagestyle{plain}
\markboth{Youngmi Hur and Zachary Lubberts
}{Generalizations of the Maillet Determinant}

\section{Introduction}

For an integer \(p>1\), we denote by \(\ZZ_{p}\) the ring of integers mod \(p\), and by \(a\,(\ZZ_{p})\) the remainder of the division of \(a\) by \(p\), i.e., for \(a=qp+r\), where \(q\in\ZZ\) and \(0\leq r<p\), \(a \,(\ZZ_{p})=r\). That is, \(a\,(\ZZ_{p})\) denotes the element \(r\) of \(\{0,1,\ldots,p-1\}\) such that \(a+p\ZZ=r+p\ZZ\). Then the Maillet determinant is defined for any odd prime \(p\) as 
$$\det\left[(i^{-1}\cdot j\,(\ZZ_{p}))\right]_{i,j=1}^{(p-1)/2},$$ where \(i^{-1}\) denotes the multiplicative inverse of \(i\) in \(\ZZ_{p}\). This determinant is mentioned in \cite{TM} as an interesting determinant and according to \cite{HWT}, E. Malo conjectured in 1913 that its value is given by the formula
$$(-p)^{{1\over 2}(p-3)}.$$
In \cite{CO},  L. Carlitz and F.R. Olson (see also \cite{CO2}) showed that the above form suggested by Malo is not exactly correct, but the determinant has the form
$$\pm p^{{1\over 2}(p-3)}k_p$$ 
with an explicitly given nonzero integer $k_p$ (that depends on $p$), which in turn implies that the determinant is always nonzero. Rather than studying the determinant of some variation of this \((p-1)/2\) by \((p-1)/2\) matrix, however, we consider instead for any odd prime \(p\) and positive integer \(m\) the determinant of the \((p-1)\) by \((p-1)\) matrix \[A_{p,m}:=\left[(i^{-1}\cdot j\,(\ZZ_{p}))^{m}\right]_{i,j=1}^{p-1}\] where the exponentiation outside of the parentheses is taken as in the integers. Then the first principal \((p-1)/2\) by \((p-1)/2\) submatrix of \(A_{p,1}\) is exactly the matrix giving rise to the Maillet determinant, and for the first few choices of \(p\) (i.e. for \(p=3,5,7\)), the matrices \(A_{p,m}\) are:
$$
\left[\begin{array}{cc} 
1&2^m\\ 
2^m&1
\end{array}\right],\quad
\left[\begin{array}{cccc} 
1&2^m&3^m&4^m\\ 
3^m&1&4^m&2^m\\
2^m&4^m&1&3^m\\
4^m&3^m&2^m&1
\end{array}\right],\quad
\left[\begin{array}{cccccc} 
1&2^m&3^m&4^m&5^m&6^m\\ 
4^m&1&5^m&2^m&6^m&3^m\\
5^m&3^m&1&6^m&4^m&2^m\\
2^m&4^m&6^m&1&3^m&5^m\\
3^m&6^m&2^m&5^m&1&4^m\\
6^m&5^m&4^m&3^m&2^m&1
\end{array}\right].
$$

In Section~\ref{S:prop}, we begin by proving several properties of the matrices \(A_{p,m}\) and their determinants; we then show that \(A_{p,m}\) is permutation similar to a circulant matrix, and discuss some properties of \(A_{p,m}\) that are implied by this fact before giving an explicit formula for the eigenvalues and eigenvectors of \(A_{p,m}\), which follows from the corresponding formulas for circulant matrices. This in turn yields a formula for the determinant of \(A_{p,m}\), an application of which we give in a lemma concerning the approximation power of the output of a new multidimensional wavelet construction (see Lemma~\ref{lem:pcs}), before concluding with a generalization of the matrices \(A_{p,m}\) to which our results can be applied (\S\ref{gen}). In particular, we will show that for \(m=1\), the determinant of \(A_{p,m}\) is always zero, and that for sufficiently large \(m\) it is always nonzero; it is our belief that the determinant of \(A_{p,m}\) never vanishes for \(m\geq 2\), though the proof eludes us.

Throughout the paper, we will denote the transpose of a matrix \(A\) by \(A^{T}\), and its conjugate transpose by \(A^{*}\).

\section{Properties of $A_{p,m}$}
\label{S:prop}
\subsection{Basic properties of $A_{p,m}$}
Each row (or column) of the matrix $A_{p,m}$ can be obtained by using a permutation operator from any other row (or column). Such a matrix is called a {\it Latin square} matrix (see, for example, \cite{KWJ} and references therein). However, the matrices $A_{p,m}$ have much more structure than the usual Latin square matrices. For example, every positive diagonal entry ($(i,i)$ for $i=1,\cdots, p-1$) is $1$ because $A_{p,m}(i,i)=(i^{-1}\cdot i\, (\ZZ_{p}))^m=1^m,$ and every anti-diagonal entry ($(i,p-i)$ for $i=1,\cdots,p-1$) is $(p-1)^m$ because $A_{p,m}(i,p-i)=(i^{-1}\cdot (p-i)\, (\ZZ_{p}))^m,$ and
$$i^{-1}\cdot (p-i) \,(\ZZ_{p})=i^{-1}\cdot p-i^{-1}\cdot i\,(\ZZ_{p})=p-1.$$

The matrix $A_{p,m}$ is {\it centrosymmetric} for every odd prime $p$ and positive integer $m$. To see why this is so, let us first recall that a matrix is called centrosymmetric if it is symmetric about its center, i.e. for even \(n\), if an $n\times n$ matrix $A$ satisfies
$A(i,j)=A(n-i+1,n-j+1)$ for $1\le i,j\le n$ (see, for example, \cite{ALA} for the definition). Now let us show that the matrix $A_{p,m}$ is centrosymmetric. For \(1\leq i\leq p-1\), we observe that \((p-i)^{-1}(\ZZ_{p})=p-i^{-1}(\ZZ_{p})\), since \[(p-i^{-1})(p-i)(\ZZ_{p})=i^{-1}\cdot i\,(\ZZ_{p})=1.\] Then for \(1\leq i,j\leq p-1\), \begin{eqnarray*}A_{p,m}(i,j)&=&i^{-1}\cdot j\,(\ZZ_{p})=(p-i^{-1})\cdot (p-j)(\ZZ_{p})\\&=&(p-i)^{-1}\cdot (p-j)(\ZZ_{p})=A_{p,m}(p-i,p-j),\end{eqnarray*} which proves that \(A_{p,m}\) is centrosymmetric.

Let us summarize what we know so far about the matrices $A_{p,m}$.
\begin{lemma} \label{lem:basic}
Let $A_{p,m}$ be defined as above for some odd prime $p$ and positive integer \(m\). Then
\begin{romannum}
\item $A_{p,m}$ is a Latin square matrix.
\item $A_{p,m}$ is centrosymmetric.
\item $A_{p,m}(i,i)=1$ and $A_{p,m}(i,p-i)=(p-1)^m$, for all $i=1,\dots,p-1$.
\end{romannum} 
\end{lemma}

Because $A_{p,m}$ is centrosymmetric, it can be written \cite{ACA} as
\begin{equation}
\label{eq:centro}
\left[\begin{array}{cc} 
B_{p,m}&JC_{p,m}J\\ 
C_{p,m}&JB_{p,m}J
\end{array}\right]
\end{equation}
where $B_{p,m}$ is the leading principal submatrix of \(A_{p,m}\) of order \((p-1)/2\), \(C_{p,m}\) is also a \((p-1)/2\) by \((p-1)/2\) submatrix, and $J$ is the the reversal matrix of the appropriate size; in this case, the $(p-1)/2$ by $(p-1)/2$ matrix of the form 
$$
\left[\begin{array}{ccccc} 
0&0&\cdots&0&1\\ 
0&0&\cdots&1&0\\ 
\vdots&\vdots&\vdots&\vdots&\vdots\\ 
0&1&\cdots&0&0\\ 
1&0&\cdots&0&0
\end{array}\right].
$$
Furthermore it is known (see \cite{MU}) that $A_{p,m}$ (cf.~(\ref{eq:centro})) is orthogonally similar to 
\begin{equation}\label{eq:centrosimilar}\left[\begin{array}{cc} 
B_{p,m}-JC_{p,m}&0\\ 
0&B_{p,m}+JC_{p,m}
\end{array}\right]\end{equation}

We know that $B_{p,m}(i,j)=(i^{-1}\cdot j\,(\ZZ_{p}))^m$ for $1\le i,j\le (p-1)/2$, and observe that for any matrix $A$, $JA$ is obtained from $A$ by flipping the order of the rows of $A$. Therefore $JC_{p,m}(i,j)=((p-i)^{-1}\cdot j\,(\ZZ_{p}))^m$, and recalling that \((p-i)^{-1}(\ZZ_{p})=p-i^{-1}(\ZZ_{p})\) (as shown above), we have that 
if $B_{p,m}(i,j)=l^m$ for some integer $l$ between $1$ and $p-1$, $JC_{p,m}(i,j)=(p-l)^m$, for every $1\le i,j\le (p-1)/2$. 

Using the properties of $A_{p,m}$ given above, among others, we get the following information about the determinant of $A_{p,m}$.
\begin{lemma}\label{lem:detinv}
Let $A_{p,m}$ be defined as above for odd prime $p$ and positive integer \(m\). Then
\begin{romannum}
\item
$\det(A_{p,m})\,(\ZZ_{4})= 0$ for $p\ge 5$.
\item
$\det A_{3,m}=1-4^m$. Hence $A_{3,m}$ is invertible for every $m\geq 1$.
\item
$\det(A_{p,m})\,(\ZZ_{p})= 0$ for every odd prime $p$ and every $m\geq 1$.
\item
$\det A_{p,1}=0$ for every prime $p\ge 5$. Hence $A_{p,1}$ is not invertible for such \(p\).
\item
$\det A_{p,m}\ne 0$ if \((p-1)^{m}-\sum_{k=1}^{p-2}k^{m}>0\). Thus, for any fixed odd prime $p$, $\det A_{p,m}$ is nonzero for sufficiently large $m$. In particular, $\det A_{p,m}$ is nonzero if $m\ge \log(p-2)\left/\log\left(\frac{p-1}{p-2}\right)\right.$. \end{romannum}
\end{lemma}

{\em Proof: }
\begin{romannum}
\item From ~(\ref{eq:centrosimilar}), we know that $\det A_{p,m}=\det(B_{p,m}-JC_{p,m})\det(B_{p,m}+JC_{p,m})$. If we write each of \(\det(B_{p,m}-JC_{p,m})\) and \(\det(B_{p,m}+JC_{p,m})\) according to the permutation definition of the determinant (i.e., for an \(n\times n\) matrix \(A\), \(\det(A)=\sum_{\sigma\in S_n}\textrm{sgn}(\sigma)\prod_{k=1}^{n}A(k,\sigma(k))\), where $S_n$ is the set of all permutations of the $n$ items $\{1,\dots,n\}$, cf. \cite{HJ}), we observe that all entries of these matrices are odd (by the explicit computation of the entries of \(B_{p,m}\) and \(JC_{p,m}\) after~(\ref{eq:centrosimilar})), so each term in the sum is odd, and since there are \((\frac{p-1}{2})!\) permutations, their sum is even. Then both factors of \(\det A_{p,m}\) are even, so \(\det A_{p,m}\) is divisible by 4.
\item Straightforward.
\item We see this by considering \(A_{p,m}(\ZZ_{p})\), and observe that each row is a scalar multiple \((i^{-1}(\ZZ_{p}))^{m}(\ZZ_{p})\) times the first row, so \(A_{p,m}\) is not invertible when considered as a matrix over \(\ZZ_{p}\), and its determinant must be \(np\) for some \(n\in\ZZ\). 
\item
From the computation after (\ref{eq:centrosimilar}), it is easy to see that when $m=1$, we have 
$$B_{p,1}(i,j)+JC_{p,1}(i,j)=p,\quad \forall i,j=1,\dots,(p-1)/2.$$
Thus for every $p\ge 5$, $\det(B_{p,1}+JC_{p,1})=0$, and $\det A_{p,1}=0$.
\item
Since the latter parts of the statement are straightforward, we show the first statement only.
For the $(p-1)$ by $(p-1)$ reversal matrix $J$, the matrix $JA_{p,m}$ is diagonally dominant if \((p-1)^{m}-\sum_{k=1}^{p-2}k^{m}>0\). Thus $\det JA_{p,m}\ne 0$ (see \cite{HJ}), which in turn implies that $\det A_{p,m}\ne 0$. 
\qquad\endproof
\end{romannum}

\subsection{$A_{p,m}$ and circulant matrices}
We recall that for an odd prime \(p\), a nonzero element $h$ of $\ZZ_{p}$ is called {\it primitive} if its powers generate the multiplicative group of \(\ZZ_{p}\), i.e., if \(\{h^{k}(\ZZ_{p}):1\leq k\leq p-1\}=\{1,2,\ldots,p-1\}\). It is well known that a primitive always exists in $\ZZ_{p}$, and, in fact, that there are Euler's totient function, \(\phi(p-1)\), many primitive elements of \(\ZZ_{p}\) (see, for example, \cite{SL}). 
Our next theorem says that $A_{p,m}$ is similar to a circulant matrix via a permutation matrix, which is defined using a primitive of $\ZZ_p$.

\begin{theorem} 
\label{thm:main}
Let \(h\) be a primitive of \(\ZZ_{p}\). Then \(A_{p,m}\) is permutation-similar to a circulant matrix whose first row is $[1,(h(\ZZ_{p}))^{m}, (h^2(\ZZ_{p}))^{m}, \dots,(h^{p-2}(\ZZ_{p}))^{m}]$. 
\end{theorem}

\begin{proof}
Denoting the standard unit vectors (as column vectors) by \(e_{j}\), \(1\leq j\leq p-1\), let \(P\) be the permutation matrix that sends \(e_{(h^{j}(\ZZ_{p}))}\) to \(e_{j}\) for each \(1\leq j\leq p-1\). Explicitly, \(P(i,j)=1\) if \(j=h^{i}(\ZZ_{p})\), and is 0 otherwise. Then \begin{eqnarray*}(PA_{p,m}P^{T})(i,j)&=&\sum_{k=1}^{p-1}(PA_{p,m})(i,k)P^{T}(k,j)=\sum_{k=1}^{p-1}A_{p,m}(h^{i}(\ZZ_{p}),k)P(j,k)\\&=&A_{p,m}(h^{i}(\ZZ_{p}),h^{j}(\ZZ_{p})),\end{eqnarray*} which is \((h^{j-i}(\ZZ_{p}))^{m}\) by the definition of \(A_{p,m}\). Since this depends only on the value of \(j-i\), we see that \(PA_{p,m}P^{T}\) is Toeplitz; but since \(h^{p-1-i}(\ZZ_{p})=h^{-i}(\ZZ_{p})=h^{1-(i+1)}(\ZZ_{p})\), \((PA_{p,m}P^{T})(i,p-1)=(PA_{p,m}P^{T})(i+1,1)\) for each \(1\leq i\leq p-2\), which proves that \(PA_{p,m}P^{T}\) is in fact a circulant matrix. 
\qquad\end{proof}

For example, let us consider the case when $p=5$. Then there are exactly two primitives, namely $2$ and $3$, in $\ZZ_5$. Following the arguments in the proof of Theorem~\ref{thm:main}, we see that the permutation matrices for $h=2$ and $h=3$ are 
\begin{equation}
\label{eq:p5permutation}
\left[\begin{array}{cccc}0&1&0&0\\0&0&0&1\\0&0&1&0\\1&0&0&0\end{array}\right],\quad\left[\begin{array}{cccc}0&0&1&0\\0&0&0&1\\0&1&0&0\\1&0&0&0\end{array}\right],
\end{equation}
respectively, and  the circulant matrices corresponding to these are those with the first rows \([1\;2^{m}\;4^{m}\;3^{m}]\), and \([1\;3^{m}\;4^{m}\;2^{m}]\), respectively. In other words, $A_{5,m}$ is permutation-similar to ${\rm Circ} (1,2^{m},4^{m},3^{m})$ via the first permutation matrix in (\ref{eq:p5permutation}) and is permutation-similar to ${\rm Circ} (1,3^{m},4^{m},2^{m})$ via the second permutation matrix in (\ref{eq:p5permutation}), where ${\rm Circ}(v)$ is defined as the circulant matrix which has \(v\) as its first row.

In fact, \(A_{p,m}\) is permutation similar to more than one circulant matrix for any \(p\geq 5\), since distinct primitives of \(\ZZ_{p}\) give rise to different corresponding permutation matrices and circulant matrices, as we will soon show. Since there are Euler's totient function, \(\phi(p-1)\), many primitive elements of \(\ZZ_{p}\) (\cite[pg. 94]{SL}), the latter statement actually implies that there are at least \(\phi(p-1)\) different circulant matrices to which \(A_{p,m}\) is permutation similar, and actually this is the exact number.\footnote{One way to prove this is to show that for any \(n\times n\) circulant matrix \(C\) and permutation matrix \(P\) such that, for \(\sigma\) a permutation of \(\{0,\ldots,n-1\}\), \(Pe_{\sigma(j)+1}=e_{j+1}\) for each \(0\leq j\leq n-1\), \(PCP^{T}\) is circulant if and only if \(\sigma\) is a homomorphism of \(\{0,\ldots,n-1\}\), considered as an additive cyclic group with generator 1. In this case, \(\sigma\) is an automorphism, and the automorphisms of a cyclic group of order \(n\) are isomorphic to the multiplicative group of \(\ZZ_{n}\) \cite{SL}, which has order \(\phi(n)\).} 

To see why \(A_{p,m}\) is permutation similar to more than one circulant matrix for any \(p\geq 5\), let us briefly use the notation \(\tensor*[^{(h)}]{P}{}\) to denote the permutation matrix obtained using \(h\) as the primitive as in the proof of Theorem~\ref{thm:main}. Observe that \(\left(\tensor*[^{(h)}]{P}{}A_{p,m}\tensor*[^{(h)}]{P}{^{T}}\right)(i,j)=(h^{j-i}(\ZZ_{p}))^{m}\), so that \(\left(\tensor*[^{(h)}]{P}{}A_{p,m}\tensor*[^{(h)}]{P}{^{T}}\right)(1,2)=h^{m}\); then by supposition, for different primitive elements \(h,\,h'\), the circulant matrices to which \(A_{p,m}\) is similar via \(\tensor*[^{(h)}]{P}{}\) and \(\tensor*[^{(h')}]{P}{}\) will differ in their \((1,2)\) entry. Note also that for \(h\) a primitive element of \(\ZZ_{p}\), \(h^{-1}(\ZZ_{p})\) is also a primitive of \(\ZZ_{p}\) but it is different from $h$ for $p\ge 5$. Hence, if we use \(h'=h^{-1}(\ZZ_{p})\) to create the permutation matrix inducing the similarity rather than \(h\), the resulting circulant matrix will be the transpose of that obtained using \(h\); i.e., 
\begin{eqnarray*}
\left(\tensor*[^{(h)}]{P}{}A_{p,m}\tensor*[^{(h)}]{P}{^{T}}\right)(i,j)&\,{=}\,&(h^{j-i}(\ZZ_{p}))^{m}=((h^{-1})^{i-j}(\ZZ_{p}))^{m}\\
&=&\left(\tensor*[^{(h')}]{P}{}A_{p,m}\tensor*[^{(h')}]{P}{^{T}}\right)(j,i)=\left(\tensor*[^{(h')}]{P}{}A_{p,m}\tensor*[^{(h')}]{P}{^{T}}\right)^{T}(i,j).
\end{eqnarray*} 
The example above illustrates this point as well, since we see that \(3=2^{-1}(\ZZ_{5})\) and \(\textrm{Circ}(1,2^m,4^m,3^m)^{T}=\textrm{Circ}(1,3^m,4^m,2^m)\).

We now list some corollaries of Theorem~\ref{thm:main}. Let us start with an immediate one.
\begin{corollary}\label{cor:norm}
$A_{p,m}$ is normal.
\end{corollary}

\begin{proof}
Since permutation matrices are unitary, $A_{p,m}$ is unitarily similar to a circulant matrix, as we saw in Theorem~\ref{thm:main}, and circulant matrices are unitarily diagonalizable. Finally, the product of two unitary matrices is unitary.
\qquad\end{proof}

We recall that circulant matrices commute because any \(n\times n\) circulant matrix is a polynomial in either the \(n\times n\) forward shift permutation matrix or equivalently, the backward shift permutation matrix (\(\textrm{Circ}(e_{2}^{T})\) or respectively \(\textrm{Circ}(e_{n}^{T})\)), the powers of which form a cyclic group of order \(n\). In light of this and Theorem~\ref{thm:main}, it comes as no surprise that the matrices \(A_{p,m}\) are actually polynomials in a single permutation matrix: \begin{corollary}\label{cor:poly}
Let $h$ be a primitive of $\ZZ_{p}$, and let $Q$ be the permutation matrix such that \(Q(i,j)=1\) if \(A_{p,m}(i,j)=h^{m}\), and is 0 otherwise. Explicitly, \(Q(i,j)=\delta(j,i\cdot h\,(\ZZ_{p}))\), where \(\delta(i,j)=1\) if \(i=j\) and is 0 otherwise. Then $$A_{p,m}=\sum_{k=1}^{p-1}(h^{k}(\ZZ_{p}))^{m}Q^{k}.$$
\end{corollary} 

\begin{proof}
Let \(P\) be the permutation matrix that sends \(e_{(h^{j}(\ZZ_{p}))}\) to \(e_{j}\) for each \(1\leq j\leq p-1\), as in the proof of Theorem \ref{thm:main}. Then \((PQP^{T})(i,j)=Q(h^{i}(\ZZ_{p}),h^{j}(\ZZ_{p}))\) by the same computation as in that proof, and this equals \(\delta(h^{j}(\ZZ_{p}),h^{i+1}(\ZZ_{p}))\). Since the powers \(h^{k}(\ZZ_{p})\) of \(h\) are distinct for \(1\leq k\leq p-1\) by definition of a primitive,  \((PQP^{T})(i,j)=1\) when \(j=i+1\) for \(1\leq i\leq p-2\) and when \(j=1\) for \(i=p-1\) --- but this is exactly \(\textrm{Circ}(e_{2}^{T})\), which we will denote \(S\) for the remainder of this proof. Since \(PA_{p,m}P^{T}\) is a circulant, we have \(PA_{p,m}P^{T}=\sum_{k=1}^{p-1}(h^{k}(\ZZ_{p}))^{m}S^{k}\) (see Theorem~\ref{thm:main} and its proof), so \(A_{p,m}=\sum_{k=1}^{p-1}(h^{k}(\ZZ_{p}))^{m}(P^{T}SP)^{k}=\sum_{k=1}^{p-1}(h^{k}(\ZZ_{p}))^{m}Q^{k}\).
\qquad\end{proof}

\noindent\emph{Remark:} \begin{romannum} 
\item We see that the powers of \(Q^{k}\) are distinct for \(1\leq k\leq p-1\), and in fact for \(1\leq k<k'\leq p-1\), \(Q^{k}\) and \(Q^{k'}\) have no common nonzero entries, since \(Q^{k}(i,j)=\delta(j,i\cdot h^{k}(\ZZ_{p}))\), as can be seen by the definition of \(Q\). Then from \(Q^{p-1}=(P^{T}SP)^{p-1}=P^{T}S^{p-1}P=I\), we see that the powers of \(Q^{k}\) form a cyclic group of order \(p-1\).
\item For positive integers \(m\) and \(m'\), \(\{A_{p,m},\,A_{p,m'},\,A_{p,m}^{T},\,A_{p,m'}^{T}\}\) is a commuting family (see also Theorem~\ref{thm:gen}(ix)). \qquad\endproof\end{romannum}

\subsection{Eigenvalues and eigenvectors of $A_{p,m}$}
Since $A_{p,m}$ is permutation-similar to a circulant matrix (cf. Theorem~\ref{thm:main}), and the eigenvalues and eigenvectors of circulant matrices are well understood (see, for example, \cite{KS}), we can write down the eigenvalues and eigenvectors of $A_{p,m}$ explicitly.

\begin{theorem} 
\label{thm:eig}
Let \(h\) be a primitive of \(\ZZ_{p}\). For $\ell=1,\dots, p-1$, let $z_\ell:=\exp(2\pi i\frac{\ell}{p-1})$ so that \(\{z_\ell:\ell=1,\dots, p-1\}\) are the distinct \((p-1)\)st roots of unity. Then the complete set of eigenvectors of \(A_{p,m}\) are 
$$\nu_{\ell}=\sum_{j=1}^{p-1}z_\ell^je_{(h^{j}(\ZZ_{p}))},\quad \ell=1,\dots, p-1$$ 
where \(e_{a}\) denotes the \(a\)th standard unit vector in $\mathbb{C}^{p-1}$ (considered as a column vector), and the eigenvalue of \(A_{p,m}\) corresponding to $\nu_\ell$ is 
$$\lambda_{\ell}=\sum_{j=1}^{p-1} z_\ell^j(h^{j}(\ZZ_{p}))^m.$$
\end{theorem}

\begin{proof} Using the well-known formula for the eigenectors of a circulant matrix, we know that the complete set of eigenvectors for \(\textrm{Circ}(e_{2}^{T})\) are \(\tilde{\nu}_{\ell}:=[z_{\ell}^{k}]_{k=1}^{p-1}\), \(1\leq \ell\leq p-1\). Then the eigenvectors of \(Q\) as in Corollary~\ref{cor:poly} are \(\nu_{\ell}=P^{T}\tilde{\nu}_{\ell}=[z_{\ell}^{\sigma(k)}]_{k=1}^{p-1}\) for \(1\leq \ell\leq p-1\), where \(\sigma(k)\) is the element \(j\in\{1,\ldots,p-1\}\) such that \(h^{j}(\ZZ_{p})=k\). Alternatively, \(\nu_{\ell}(h^{k}(\ZZ_{p}))=z_{\ell}^{k}\), as in the statement. By computation, \(Q\nu_{\ell}=z_{\ell}\nu_{\ell}\), so we have for any polynomial \(r\), \(r(Q)\nu_{\ell}=r(z_{\ell})\nu_{\ell}\). Then in particular, \(A_{p,m}\nu_{\ell}=\lambda_{\ell}\nu_{\ell}\), with \(\lambda_{\ell}\) as given above, using the polynomial found in Corollary~\ref{cor:poly}.\qquad\end{proof}

The next corollary of Theorem~\ref{thm:eig} says that the eigenvectors $\nu_\ell$ found in the theorem are either symmetric, i.e. $J\nu_\ell=\nu_\ell$, or skew-symmetric, i.e. $J\nu_\ell=-\nu_\ell$, depending on the parity of $\ell$. Here, $J$ is used to denote the reversal matrix as before.
\begin{corollary}
\label{cor:symm}
For even \(\ell\), \(J\nu_{\ell}=\nu_{\ell}\) holds, and for odd \(\ell\), \(J\nu_{\ell}=-\nu_{\ell}\) holds.
In particular, there is an eigenvector $v$ of \(A_{p,m}\) with all entries in \(\{1,-1\}\), which satisfies $Jv=v$ if \(p\,(\ZZ_{4})=1\), and $Jv=-v$ if \(p\,(\ZZ_{4})=3\).
\end{corollary}

\begin{proof} 
We first make a simple observation: 
\begin{equation}
\label{eq:relation}
(h^{k}(\ZZ_{p}))+(h^{(p-1)/2+k}(\ZZ_{p}))=p,\textrm{ for }k=1,\ldots,(p-1)/2.
\end{equation}
Using this, we see that \begin{eqnarray*}J\nu_{\ell}&=&\sum_{k=1}^{p-1}z_{\ell}^{k}e_{p-(h^{k}(\ZZ_{p}))}=\sum_{k=1}^{p-1}z_{\ell}^{k}e_{(h^{(p-1)/2+k}(\ZZ_{p}))}\\&=&\sum_{k=1}^{p-1}z_{\ell}^{k-(p-1)/2}e_{(h^{k}(\ZZ_{p}))}=z_{\ell}^{-(p-1)/2}\nu_{\ell}.\end{eqnarray*} 
The first statement follows since \(z_{\ell}^{-(p-1)/2}=\mathrm{exp}\left(\frac{2\pi i\ell}{p-1}\left(-\frac{p-1}{2}\right)\right)=\mathrm{exp}(-\pi i\ell)=(-1)^{\ell}\).
The second statement follows from the first one by considering $\ell=(p-1)/2$ and the eigenvector $v:=\nu_{(p-1)/2}=\sum_{j=1}^{p-1}\exp(ij\pi)e_{(h^{j}(\ZZ_{p}))}$.
\qquad\end{proof}

\noindent\emph{Remark}: We recall that for a square matrix \(A\) of order \(n\), where \(n\) is even, being centrosymmetric is equivalent to commuting with \(J\). Hence if, in addition to being centrosymmetric, \(A\) is diagonalizable, then $A$ and $J$ are simultaneously diagonalizable, which means that there exists an invertible matrix $V=[v_1,\dots,v_n]$ such that $AV=V\Lambda_A$ and $JV=V\Lambda_J$, where $v_i$'s are common eigenvectors of $A$ and $J$, and $\Lambda_A$ and $\Lambda_J$ are diagonal matrices consisting of eigenvalues of $A$ and $J$, respectively. Since it is easy to see that $$\Lambda_J=\left[\begin{array}{cc}I&0\\0&-I\end{array}\right],$$
we obtain that $Jv_i=v_i$ for $i=1,\dots,{n\over 2}$, and $Jv_i=-v_i$ for $i={n\over 2}+1,\dots,n$.
Thus, if $A$ is centrosymmetric and diagonalizable, then any matrix that diagonalizes \(A\) must have \(n/2\) columns spanning \(\{v\in\mathbb{C}^{n}:Jv=v\}\), and \(n/2\) columns spanning \(\{v\in\mathbb{C}^{n}:Jv=-v\}\).

Since our matrix \(A_{p,m}\) is diagonalizable and centrosymmetric, the above holds for \(A_{p,m}\), hence the main contribution  made by the preceding corollary is in identifying the symmetric and skew-symmetric eigenvectors among \(\nu_{\ell},\,\ell=1,\ldots,p-1\).\qquad\endproof

Recall from Lemma~\ref{lem:detinv}(iv) that, for every prime $p\ge 5$, $\det A_{p,1}=0$, i.e., $\lambda=0$ is an eigenvalue of $A_{p,1}$. In fact, by inspection of the proof there, it is clear that 0 is an eigenvalue of algebraic multiplicity at least \((p-1)/2-1\). We now use Theorem~\ref{thm:eig} to find \((p-1)/2-1\) linearly independent eigenvectors among the \(\nu_{\ell},\,1\leq\ell\leq p-1\), corresponding to the eigenvalue $\lambda=0$ when \(m=1\), which proves that the geometric multiplicity of 0 is also at least \((p-1)/2-1\) in this case.

\begin{corollary} 
\label{coro:mone}
Let \(\ell\in\{1,\ldots,p-1\}\) be even and \(\ell\neq p-1\). Then the eigenvalue \(\lambda_{\ell}\) of \(A_{p,1}\) is equal to 0.\end{corollary}

\begin{proof} 
Let us denote the greatest common divisor of two positive integers \(a\) and \(b\) by 
\(\mathrm{gcd}(a,b)\). Since \(\ell\) 
is even, 
\(\mathrm{gcd}(\ell,p-1)=d\), \(p-1=fd\), \(\ell=gd\), 
so that 
\(\mathrm{gcd}(f,g)=1\)
and 
\(2|d\). 
Observing that \(\exp\left(\frac{2\pi ik\ell}{p-1}\right)=\exp\left(\frac{2\pi ikg}{f}\right)\) is \(f\)-periodic in \(k\), we may write \[\lambda_{\ell}=\sum_{k=1}^{p-1}(h^{k}(\ZZ_{p}))z_{\ell}^{k}=\sum_{k=1}^{f}\left(\sum_{j=1}^{d}(h^{f(j-1)+k}(\ZZ_{p}))\right)\exp\left(\frac{2\pi ikg}{f}\right).\] Also, note that \(f(d/2)=\frac{p-1}{2}\). Then 
$$
\sum_{j=1}^{d}h^{f(j-1)+k}(\ZZ_{p})=\sum_{j=1}^{d/2}(h^{f(j-1)+k}(\ZZ_{p}))+(h^{f(d/2+j-1)+k}(\ZZ_{p}))=\frac{d}{2}p,
$$ 
where the last line follows from the relation (\ref{eq:relation}) mentioned in the proof of Corollary \ref{cor:symm}. Thus \(\lambda_{\ell}=\frac{dp}{2}\sum_{k=1}^{f}\exp\left(\frac{2\pi ikg}{f}\right)\), and since \(\mathrm{gcd}(f,g)=1\), \(\exp\left(\frac{2\pi ikg}{f}\right)\) is a primitive \(f\)th root of unity \cite{SL}, so \(\lambda_{\ell}=0\). Observe that when \(\ell=p-1\), \(z_{\ell}=1\), so \(\lambda_{\ell}=\sum_{k=1}^{p-1}k^{m}\neq 0\). \qquad\end{proof}\\

\noindent\emph{Remark}: It is easy to find \((p-1)/2-1\) real eigenvectors of \(A_{p,1}\) with eigenvalue 0, and in fact these have a nice structure. For a prime \(p\geq 5\), let \(C\in\mathbb{R}^{(p-1)/2\times(p-1)/2-1}\) be the matrix with columns \(e_{1}-e_{k}\), \(2\leq k\leq (p-1)/2\). Then \[A_{p,1}\left[\begin{array}{c}C\\JC\end{array}\right]=0.\] This follows since for \(2\leq k\leq p-1\), \((A_{p,1}(e_{1}+e_{p-1}-e_{k}-e_{p-k}))(i)=(i^{-1}(\ZZ_{p}))+(p-i^{-1}(\ZZ_{p}))-(i^{-1} k\,(\ZZ_{p}))-(p-i^{-1} k\,(\ZZ_{p}))=p-p=0\), for all $i=1,\dots,p-1$.\qquad\endproof

Since we know the eigenvalues of $A_{p,m}$ (cf. Theorem~\ref{thm:eig}), its determinant can be written explicitly as follows.
\begin{corollary} 
\label{cor:det}
The determinant of \(A_{p,m}\) is given by \[\prod_{\ell=1}^{p-1}\left(\sum_{k=1}^{p-1}z_\ell^k(h^{k}(\ZZ_{p}))^{m}\right).\]
\end{corollary}

Corollary~\ref{cor:det} provides an explicit formula for the determinant of $A_{p,m}$, so one can always compute the determinant of $A_{p,m}$ for any fixed $p$ and $m$. This result can be considered complementary to Lemma~\ref{lem:detinv}, where we collected information about the determinant of $A_{p,m}$ for some special cases. Knowing the determinant $A_{p,m}$, or the invertibility of $A_{p,m}$, can be useful in solving some other  problems. 

For example, the invertibility of $A_{p,m}$ can be used to answer a question about the approximation power of a new multidimensional wavelet system constructed in \cite{PCS}. In fact, this connection prompted our interest in the matrices $A_{p,m}$, hence we mention it very briefly below, leaving all the technical details to the article.

Let $p$ be an odd prime and $m$ a positive integer. Let $R(\xi):=\sum_{k\in\ZZ}r(k)e^{-ik\xi}$, where $r(k)$ is nonzero for only finitely many $k$, and $\xi\in[-\pi,\pi]$. Suppose that $R(0)=1$ and $R$ has a zero of order $m$ at every point in ${2\pi\over p}\{1,\cdots,p-1\}$. For $\dm\ge 2$, define $\tau$ as
$$\tau(\ome):={1\over
(p-1)p^{\dm-1}}\left(1-p^{\dm-1}
+\sum_{\nu\in\{0,1,\cdots,p-1\}^\dm\bks0}R(\ome\cdot\nu)\right),\quad\ome\in[-\pi,\pi]^\dm.$$
A proof that \(\tau\) is well-defined can be found in \cite{PCS}. Then one has the following result.

\begin{lemma}
\label{lem:pcs}
If $A_{p,m}$ is invertible, then the order of zeros of $\tau$ at every point in ${2\pi\over p}\{0,1,\cdots,p-1\}^\dm\bks0$ is at most $m$, regardless of $\dm$.
\end{lemma}

\begin{proof}
We will denote the partial derivative operator with \(\alpha(i)\) partial derivatives on \(\omega(i)\) for each \(1\leq i\leq n\) for some \(\alpha\in\ZZ^{n}\) such that \(\alpha(i)\geq 0,\,1\leq i\leq n\) by \(D^{\alpha}\), and the univariate \(k\)th derivative operator by \(D^{k}\). Let $\Gamma=\{0,1,2,\ldots,p-1\}^{n}\setminus0$. Then clearly there are \(p^{n}-1\) vectors \(\nu\in\Gamma\), and of these,
\(p^{n-1}-1\) have \(\nu(1)=0\), and \(p^{n-1}\) have \(\nu(1)=k\) for each \(k=1,2,\ldots,p-1\). Then for \(\xi\in[-\pi,\pi]\),
\begin{eqnarray*}
\tau(\xi,0,\ldots,0)&=&\frac{1}{(p-1)p^{n-1}}\left(1-p^{n-1}+\sum_{\nu\in\Gamma}R(\xi\nu(1))\right)\\&=&\frac{1}{(p-1)p^{n-1}}\left(1-p^{n-1}+(p^{n-1}-1)R(0)+p^{n-1}\sum_{k=1}^{p-1}R(k\xi)\right).\end{eqnarray*} 
Since \(R(0)=1\), this equals \(\frac{1}{p-1}\sum_{k=1}^{p-1}R(k\xi)\). 
Then \[D^{(m,0,\ldots,0)}\tau(\xi,0,\ldots,0)=\frac{1}{p-1}\sum_{k=1}^{p-1}D^{m}[R(k\xi)]=\frac{1}{p-1}\sum_{k=1}^{p-1}k^{m}[D^{m}R](k\xi).\] 
Evaluating at \(\xi=\frac{2\pi\ell}{p}\), we obtain 
\[\frac{1}{p-1}\sum_{k=1}^{p-1}k^{m}[D^{m}R]\left(\frac{2\pi k\ell}{p}\right)=\frac{1}{p-1}\sum_{k=1}^{p-1}(k\cdot\ell^{-1}(\ZZ_{p}))^{m}[D^{m}R]\left(\frac{2\pi k}{p}\right).\] Since the vector \(v:=[(D^{m}R)(2\pi k/p)]_{k=1}^{p-1}\) is nonzero because $R$ has a zero of order $m$ at every point in ${2\pi\over p}\{1,\cdots,p-1\}$, if \(A_{p,m}\) is invertible, then \(\frac{1}{p-1}A_{p,m} v\neq 0\). But since the computations above show that \[\frac{1}{p-1}(A_{p,m}v)(\ell)=D^{(m,0,\ldots,0)}\tau\left(\frac{2\pi\ell}{p},0,\ldots,0\right),\] the right hand side of this equation is nonzero for some \(\ell\), and this proves the desired statement.
\qquad\end{proof}

\subsection{Generalizations of $A_{p,m}$}
\label{gen}
While we have focused on the matrices \(A_{p,m}\) with the specific entries given in their definition, we see that we can obtain a large class of Latin square matrices to which our results may be applied by considering matrices that have the same permutation structure as \(A_{p,m}\), but possibly different entries. In particular, given a vector \(c\in\mathbb{C}^{p-1}\), we define \begin{equation}\label{eq:apc}A_{p}[c]=\left[c(i^{-1}\cdot j\,(\ZZ_{p}))\right]_{i,j=1}^{p-1}.\end{equation} Then \(A_{p,m}=A_{p}[k^{m}]_{k=1}^{p-1}\), and for any \(c\in\mathbb{C}^{p-1}\), \(A_{p}[c]\) is the Latin square matrix with the same structure for the locations of its entries as \(A_{p,m}\), but with first row given by the vector \(c\). For a fixed primitive \(h\) of \(\ZZ_{p}\), we will also define the polynomial associated with \(A_{p}[c]\), \(f_{c}(z)=\sum_{k=1}^{p-1}c(h^{k}(\ZZ_{p}))z^{k}\), where we suppress the dependence on \(h\) from the notation. Observe that \(A_{p}[c]\) may also be written as \(f_{c}(Q)\), where \(Q\) is the permutation matrix given in Corollary~\ref{cor:poly} for the same primitive \(h\) used to define \(f_{c}\). We collect the results about \(A_{p}[c]\) in the theorem below:
\begin{theorem}\label{thm:gen} Let \(c\in\mathbb{C}^{p-1}\), and let \(A_{p}[c]\) be given as in (\ref{eq:apc}). Also, fix a primitive \(h\) in \(\ZZ_{p}\), and let \(f_{c}(z)\) be the polynomial associated with \(A_{p}[c]\). Then \begin{romannum} \item \(A_{p}[c]\) is a Latin square matrix if the entries of \(c\) are distinct.
\item \(A_{p}[c]\) is centrosymmetric.
\item \(A_{p}[c]\) has constant diagonal \(c(1)\) and antidiagonal \(c(p-1)\).
\item \(A_{p}[c]\) is permutation similar to \(\textrm{Circ}(c(h^{k-1}(\ZZ_{p})))_{k=1}^{p-1}\).
\item \(A_{p}[c]\) is normal.
\item The eigenvectors of \(A_{p}[c]\) are given by \(\nu_{\ell}\) as in Theorem~\ref{thm:eig}, with corresponding eigenvalues \(\lambda_{\ell}(c)=f_{c}(z_{\ell})\), where $z_\ell=\exp(2\pi i\frac{\ell}{p-1})$ as before. 
\item Corollary~\ref{cor:symm} holds with \(A_{p,m}\) replaced by \(A_{p}[c]\) in its statement.
\item The determinant of \(A_{p}[c]\) is given by \(\prod_{\ell=1}^{p-1}f_{c}(z_{\ell})\). In particular, \(A_{p}[c]\) is invertible provided that \(f_{c}\) has no roots among the \((p-1)\)st roots of unity.
\item For any \(c,c'\in\mathbb{C}^{p-1}\), \(\{A_{p}[c],A_{p}[c'],A_{p}[c]^{*},A_{p}[c']^{*}\}\) is a commuting family.\end{romannum}\end{theorem}
\begin{proof} (i)-(iii) are clear from (\ref{eq:apc}) and the discussion preceding Lemma~\ref{lem:basic}. To prove (iv), let \(P\) be the permutation matrix such that \(Pe_{(h^{j}(\ZZ_{p}))}=e_{j}\), \(1\leq j\leq p-1\), as in the proof of Theorem~\ref{thm:main}. Then following the computation there, \((PA_{p}[c]P^{T})(i,j)=c(h^{j-i}(\ZZ_{p}))\). It follows that \(PA_{p}[c]P^{T}\) is circulant in exactly the same way it is proven for \(A_{p,m}\) there. The proof of (v) is the same as that of Corollary~\ref{cor:norm} with \(A_{p}[c]\) replacing \(A_{p,m}\) throughout, and (iv) replacing Theorem~\ref{thm:main}. (vi) follows immediately from the last line of the proof of Theorem~\ref{thm:eig}, with \(r=f_{c}\), and (vii) and (viii) are obvious. (ix) follows since all four of these matrices are polynomials in \(Q\).
\qquad\end{proof}

\section{Summary and Overview}
In this paper, we studied the determinant of an extension of the matrix underlying the Maillet determinant. We showed that the matrix $A_{p,m}$ has many interesting properties, e.g. it is a Latin square, centrosymmetric, normal, and permutation-similar to a circulant matrix. Using these properties, we obtained a formula for the determinant of $A_{p,m},$ as well as the much broader collection of matrices \(A_{p}[c]\). In investigating the properties of the matrix $A_{p,m}$, the group structure of \(\ZZ_{p}\setminus\{0\}\) is used strongly. This is by no means surprising since this group is used to define $A_{p,m}$, but its ubiquitous presence (along with that of many other groups isomorphic to it) is both very interesting and something we did not expect.

Despite reporting many interesting properties of $A_{p,m}$ in this paper, there are sure to be many others that have not yet been discovered. We plan to continue investigating $A_{p,m}$, and hopefully to find more of these properties. In particular, we do not yet know whether $A_{p,m}$ is invertible for every prime $p\ge 5$ and integer $m\ge 2$, though this is the authors' belief.


\end{document}